\documentclass{amsart}

\newtheorem{theorem}{Theorem}[section]
\newtheorem{lemma}[theorem]{Lemma}

\theoremstyle{definition}
\newtheorem{definition}[theorem]{Definition}
\newtheorem{example}[theorem]{Example}

\theoremstyle{remark}

\numberwithin{equation}{section}

\begin{document}

\title[] {Some new properties of Confluent Hypergeometric Functions}

\author{Xu-Dan LUO }
\address{Department of Mathematics, Hong Kong University of
Science and Technology, Clear Water Bay, Kowloon, Hong Kong,\ P. R. China}
\email{xluoab@ust.hk}
\thanks{}

\author{Wei-Chuan LIN}
\address{Department of Mathematics, Fujian Normal University,\ P. R. China}
\email{lwc@fjnusoft.com}
\thanks{This work was supported by the National Natural Science Foundation of China(11371225) and the Natural Science Foundation of Fujian Province in China (2011J01006).}

\subjclass[2000]{Primary : 33C15, 30D30, 30D35
}

\keywords{Confluent hypergeometric functions, Nevanlinna theory, Wiman-Valiron theory}

\begin{abstract}
The confluent hypergeometric functions (the Kummer functions) defined by
${}_{1}F_{1}(\alpha;\gamma;z):=\sum_{n=0}^{\infty}\frac{(\alpha)_{n}}{n!(\gamma)_{n}}z^{n}\ (\gamma\neq 0,-1,-2,\cdots)$,
which are of many properties and great applications in statistics, mathematical physics, engineering and so on, have been given.
In this paper, we investigate some new properties of ${}_{1}F_{1}(\alpha;\gamma;z)$ from the perspective of value distribution theory. Specifically, two different growth orders are obtained for $\alpha\in \mathbb{Z}_{\leq 0}$ and $\alpha\not\in \mathbb{Z}_{\leq 0}$, which are corresponding to the reduced case and non-degenerated case of ${}_{1}F_{1}(\alpha;\gamma;z)$. Moreover, we get an asymptotic estimation of characteristic function $T(r,{}_{1}F_{1}(\alpha;\gamma;z))$ and a more precise result of $m\left(r, \frac{{}_{1}F_{1}'(\alpha;\gamma;z)}{{}_{1}F_{1}(\alpha;\gamma;z)}\right)$, compared with the Logarithmic Derivative Lemma. Besides, the distribution of zeros of the confluent hypergeometric functions is discussed. Finally, we show how a confluent hypergeometric function and an entire function are uniquely determined by their $c$-values.
\end{abstract}

\maketitle



\section{Introduction}



\subsection{Background}

It is a well-known fact that some special functions are the solutions of second-order differential equations, whose standard form is
\begin{equation}
\label{E:1}
\frac{d^{2}y}{dz^{2}}+p(z)\frac{dy}{dz}+q(z)y(z)=0,
\end{equation}
where $p(z)$ and $q(z)$ are given complex-valued functions.

The singularities of (\ref{E:1}) may be those of solutions, and solutions are analytic at the ordinary points of (\ref{E:1}).

\begin{definition}
If $z=z_{0}$ is a singularity of (\ref{E:1}), then $z=z_{0}$ is a regular singularity of (\ref{E:1}) if and only if
\begin{equation}
(z-z_{0})\cdot p(z), \ \ (z-z_{0})^{2}\cdot q(z)
\end{equation}
are analytic in $\{z: |z-z_{0}|<R\}$, where $R$ is a positive real number.
\end{definition}

\begin{definition}
A class of equations with $n$ regular singularities is called Fuchsian type equations, where $n$ is a positive integer.
\end{definition}

When $n=3$, the Fuchsian type equation
\begin{equation}
\label{E:2}
z(1-z)\frac{d^{2}y}{dz^{2}}+\left[\gamma-\left(\alpha+\beta+1\right)z\right]\frac{dy}{dz}-\alpha\beta y(z)=0
\end{equation}
is called the Gauss hypergeometric equation, which has three regular singularities $\{0,1,\infty\}$.
\begin{equation}
\label{E:3}
{}_{2}F_{1}\left(\alpha,\beta;\gamma;z\right):=\sum_{n=0}^{\infty}\frac{(\alpha)_{n}(\beta)_{n}}{n!(\gamma)_{n}}z^{n}\ (\gamma\neq 0,-1,-2,\cdots)
\end{equation}
satisfies (\ref{E:2}) and is convergent in $D(0,1):=\{z:|z|<1\}$. When $\gamma\not\in \mathbb{Z}$,
$z^{1-\gamma}{}_{2}F_{1}(\alpha-\gamma+1,\beta-\gamma+1;2-\gamma;z)$ is linearly independent of (\ref{E:3}) and convergent in $D(0,1)$, where $(\alpha)_{n}:=\alpha(\alpha+1)\cdots(\alpha+n-1)=\frac{\Gamma(\alpha+n)}{\Gamma(\alpha)}$ is called rising factorial of length $n$.

Moreover, each equation with three
regular singularities can be transferred into the Gauss hypergeometric equation, whose solutions are related to the hypergeometric functions ${}_{2}F_{1}\left(\alpha,\beta;\gamma;z\right)$.

If we replace $z$ by $\frac{z}{b}$ in the Gauss hypergeometric equation (\ref{E:2}), let $b=\beta\rightarrow \infty$, then (\ref{E:2}) will be reduced into the confluent hypergeometric equation
\begin{equation}
\label{E:4}
z\frac{d^{2}y}{dz^{2}}+(\gamma-z)\frac{dy}{dz}-\alpha y(z)=0.
\end{equation}
It is easy to see that $z=0$ is still a regular singularity, but $z=\infty$ becomes an irregular singularity. Around the origin,
\begin{equation}
{}_{1}F_{1}(\alpha;\gamma;z):=\sum_{n=0}^{\infty}\frac{(\alpha)_{n}}{n!(\gamma)_{n}}z^{n}\ (\gamma\neq 0,-1,-2,\cdots)
\end{equation}
is one solution of (\ref{E:4}), which is entire, called the confluent hypergeometric function (the Kummer function).
In particular, when $\alpha=-n=0,-1,-2,\cdots$, ${}_{1}F_{1}(-n;\gamma;z)$ is a polynomial.

On the other hand, from the viewpoint of physics (see \cite{Slater}, Chapter 1), physicists started to focus on the confluent hypergeometric functions when they tried to solve equations of the type
\begin{equation}
\bigtriangledown^{2}\psi+\frac{8\pi^{2}M}{h^{2}}\left(E+\frac{\mu}{r}\right)\psi=0,
\end{equation}
where $\bigtriangledown^{2}=\frac{\partial^{2}}{\partial x^{2}}+\frac{\partial^{2}}{\partial y^{2}}+\frac{\partial^{2}}{\partial z^{2}}$, $\psi$ is a function of the spherical coordinates $r$, $\theta$ and $\phi$.

Then $\psi=\alpha(r) \beta(\theta) \gamma(\phi)$, where $\gamma(\phi)=e^{mi\phi}$, $\beta(\theta)=P_{l}^{m}(\cos \theta)$, a Legendre polynomial, and $\alpha(r)$ is a function satisfies a second-order differential equation
\begin{equation}
r^{2}\alpha''(r)+2r\alpha'(r)+(ar^{2}+br)R(r)=c,
\end{equation}
whose solutions are related to the confluent hypergeometric functions.

\subsection{The general solutions of confluent hypergeometric equations near the origin} (see \cite{Bateman})
If we put $z=\lambda \xi$, $y=z^{\rho}\cdot e^{hz}\cdot \eta$, the confluent hypergeometric equation (\ref{E:4}) transforms into
\begin{equation}
\begin{split}
&\xi \frac{d^{2}\eta}{d\xi^{2}}+\left[\gamma+2\rho-(1-2h)\lambda \xi\right]\frac{d\eta}{d\xi}\\
&+\left[\frac{\rho(\rho+\gamma-1)}{\xi}-\lambda(\alpha-h\gamma+\rho-2h\rho)+\lambda^{2}h(h-1)\xi\right]\eta=0.
\end{split}
\end{equation}
This equation will have the same form of (\ref{E:4}) if $\rho=0$ or $\rho=1-\gamma$, $h=0$ or $h=1$, $\lambda(1-2h)=1$.

Since one solution of (\ref{E:4}) is
\begin{equation}
y_{1}={}_{1}F_{1}(\alpha;\gamma;z),
\end{equation}
and the transformations give three further solutions,
\begin{equation}
y_{2}=z^{1-\gamma}{}_{1}F_{1}(\alpha-\gamma+1;2-\gamma;z),
\end{equation}
\begin{equation}
y_{3}=e^{z}{}_{1}F_{1}(\gamma-\alpha;\gamma;-z),
\end{equation}
\begin{equation}
y_{4}=z^{1-\gamma}e^{z}{}_{1}F_{1}(1-\alpha;2-\gamma;-z).
\end{equation}
From their behavior at the origin, it follows that $y_{1}$ and $y_{2}$ are linearly independent if $\gamma\not\in \mathbb{Z}$ so that (in this case) the general solution of (\ref{E:4}) may be written as
$y=Ay_{1}+By_{2}$, where $A$ and $B$ are constants.

Note that both $y_{1}$ and $y_{3}$ are solutions of (\ref{E:4}) and regular at the origin, having the value unity there. If $\gamma\not\in \mathbb{Z}$, differential equation (\ref{E:4}) cannot have more than one such solution, then we must have $y_{1}=y_{3}$, i.e.,
\begin{equation}
{}_{1}F_{1}(\alpha;\gamma;z)=e^{z}{}_{1}F_{1}(\gamma-\alpha;\gamma;-z),
\end{equation}
which is known as Kummer's transformation. Similarly, $y_{2}=y_{4}$ can be obtained by Kummer's transformation also.

\subsection{Some properties of the confluent hypergeometric functions} (see \cite{Bateman} and \cite{Slater})
For each confluent hypergeometric function ${}_{1}F_{1}(\alpha;\gamma;z)$, there exist four functions
\begin{equation}
{}_{1}F_{1}(\alpha+1;\gamma;z), \ \ {}_{1}F_{1}(\alpha-1;\gamma;z), \ \ {}_{1}F_{1}(\alpha;\gamma+1;z), \ \ {}_{1}F_{1}(\alpha;\gamma-1;z)
\end{equation}
contiguous to it are linearly connected. They form the following recurrence relations.
\begin{equation}
(\gamma-\alpha){}_{1}F_{1}(\alpha-1;\gamma;z)+(2\alpha-\gamma+z){}_{1}F_{1}(\alpha;\gamma;z)-\alpha {}_{1}F_{1}(\alpha+1;\gamma;z)=0,
\end{equation}
\begin{equation}
\gamma(\gamma-1){}_{1}F_{1}(\alpha;\gamma-1;z)-\gamma(\gamma-1+z){}_{1}F_{1}(\alpha;\gamma;z)+z(\gamma-\alpha) {}_{1}F_{1}(\alpha;\gamma+1;z)=0,
\end{equation}
\begin{equation}
(\alpha-\gamma+1){}_{1}F_{1}(\alpha;\gamma;z)-\alpha {}_{1}F_{1}(\alpha+1;\gamma;z)+(\gamma-1) {}_{1}F_{1}(\alpha;\gamma-1;z)=0,
\end{equation}
\begin{equation}
\gamma {}_{1}F_{1}(\alpha;\gamma;z)-\gamma {}_{1}F_{1}(\alpha-1;\gamma;z)-z {}_{1}F_{1}(\alpha;\gamma+1;z)=0,
\end{equation}
\begin{equation}
\gamma (\alpha+z){}_{1}F_{1}(\alpha;\gamma;z)+z(\alpha-\gamma) {}_{1}F_{1}(\alpha;\gamma+1;z)-\alpha \gamma {}_{1}F_{1}(\alpha+1;\gamma;z)=0,
\end{equation}
\begin{equation}
(\alpha-1+z) {}_{1}F_{1}(\alpha;\gamma;z)+(\gamma-\alpha) {}_{1}F_{1}(\alpha-1;\gamma;z)+(1-\gamma) {}_{1}F_{1}(\alpha;\gamma-1;z)=0.
\end{equation}
These relations are not all independent. If we choose two of them suitably, all the others follow by simple operations.

Any function ${}_{1}F_{1}(\alpha+m;\gamma+n;z)$, $m$, $n$ integers, is said to be associated with ${}_{1}F_{1}(\alpha;\gamma;z)$. By repeated application of the relations between contiguous functions, any three associated functions are connected by a homogeneous linear relation whose coefficients are polynomials in $z$. Moreover,
\begin{equation}
\label{E:differentiation formula}
\frac{d^{n}}{dz^{n}}{}_{1}F_{1}(\alpha;\gamma;z)=\frac{(\alpha)_{n}}{(\gamma)_{n}}{}_{1}F_{1}(\alpha+n;\gamma+n;z).
\end{equation}

\subsection{Integral representations of the confluent hypergeometric functions} (see \cite{Bateman} and \cite{Wang and Guo})
It is known that homogeneous linear differential equations whose coefficients are linear functions of the independent variable can be integrated by Laplace integrals. Then
\begin{equation}
{}_{1}F_{1}(\alpha;\gamma;z)=\frac{\Gamma(\gamma)}{\Gamma(\alpha)\Gamma(\gamma-\alpha)}
\int_{0}^{1}e^{zu}u^{\alpha-1}(1-u)^{\gamma-\alpha-1}du,
\end{equation}
where $\Re \gamma>\Re \alpha >0$, $\arg u=\arg (1-u)=0$.

Another type of integral representation uses Mellin-Barnes integrals, thus
\begin{equation}
{}_{1}F_{1}(\alpha;\gamma;z)=\frac{1}{2\pi i}\cdot\frac{\Gamma(\gamma)}{\Gamma(\alpha)}\cdot
\int_{-i\infty}^{i\infty}\frac{\Gamma(\alpha+s)\Gamma(-s)}{\Gamma(\gamma+s)}(-z)^{s}ds,
\end{equation}
where $\alpha\neq 0,-1,-2,\cdots$, $|\arg(-z)|<\frac{\pi}{2}$.

Besides,
\begin{equation}
{}_{1}F_{1}(\alpha;\gamma;z)=\frac{1}{(2\pi i)^{2}}\cdot e^{-i\pi \gamma}\Gamma(1-\alpha)\Gamma(\gamma)\Gamma(1+\alpha-\gamma)
\int^{(1+,0+,1-,0-)}e^{zt}t^{\alpha-1}(1-t)^{\gamma-\alpha-1}dt,
\end{equation}
where the contour of integration is a double loop starting at a point $A$ between $0$ and $1$ on the real $t$ axis,
with $\arg t=\arg(1-t)=0$ at $A$, encircling first $t=1$ in the positive sense, then $t=0$ in the positive sense, then
$t=1$ in the negative sense, and finally $t=0$ in the negative sense, returning to $A$.

\begin{equation}
{}_{1}F_{1}(\alpha;\gamma;z)=\frac{1}{2\pi i}\cdot\frac{\Gamma(\gamma)\Gamma(\alpha-\gamma+1)}{\Gamma(\alpha)}
\int_{0}^{(1+)}e^{zt}t^{\alpha-1}(t-1)^{\gamma-\alpha-1}dt,
\end{equation}
where $\Re \alpha>0$, the contour is a loop starting (and ending) at $t=0$ and encircling $1$ once in the positive sense.

\begin{equation}
{}_{1}F_{1}(\alpha;\gamma;z)=-\frac{1}{2\pi i}\cdot\frac{\Gamma(\gamma)\Gamma(1-\alpha)}{\Gamma(\gamma-\alpha)}
\int_{1}^{(0+)}e^{zt}t^{\alpha-1}(1-t)^{\gamma-\alpha-1}dt,
\end{equation}
where $\Re (\gamma-\alpha)>0$, the contour is the above one by interchanging the roles of $0$ and $1$.

\subsection{Asymptotic behavior} (see \cite{Bateman})
The asymptotic behavior of confluent hypergeometric functions is different according as the large quantity is the variable, one of the parameters, or two or all three of these quantities.

Asymptotic behavior for large $|z|$ is
\begin{equation}
\label{E:asymptotic-1}
\begin{split}
&{}_{1}F_{1}(\alpha;\gamma;z)=\frac{\Gamma(\gamma)}{\Gamma(\gamma-\alpha)}\left(\frac{e^{i\pi \epsilon}}{z}\right)^{\alpha}
\sum_{n=0}^{M}\frac{(\alpha)_{n}(\alpha-\gamma+1)_{n}}{n!}(-z)^{-n}+O(|z|^{-\alpha-M-1})\\
&+\frac{\Gamma(\gamma)}{\Gamma(\alpha)}e^{z}z^{\alpha-\gamma}\sum_{n=0}^{N}\frac{(\gamma-\alpha)_{n}(1-\alpha)_{n}}{n!}
z^{-n}+O(\left|e^{z}z^{\alpha-\gamma-N-1}\right|),
\end{split}
\end{equation}
where $M,N=0,1,2,\cdots$, $\epsilon=1$ if $\Im z>0$, $\epsilon=-1$ if $\Im z<0$, $z\rightarrow \infty$,
$-\pi<\arg z< \pi$.

In particular, as $\Re z\rightarrow\infty$,
\begin{equation}
\label{E:asymptotic-2}
{}_{1}F_{1}(\alpha;\gamma;z)=\frac{\Gamma(\gamma)}{\Gamma(\alpha)}e^{z}z^{\alpha-\gamma}\left[1+O(|z|^{-1})\right],
\end{equation}
and as $\Re z\rightarrow -\infty$,
\begin{equation}
\label{E:asymptotic-3}
{}_{1}F_{1}(\alpha;\gamma;z)=\frac{\Gamma(\gamma)}{\Gamma(\gamma-\alpha)}(-z)^{-\alpha}\left[1+O(|z|^{-1})\right].
\end{equation}
Moreover, if $\alpha$ and $z$ are bounded, $\gamma\rightarrow \infty$,
\begin{equation}
{}_{1}F_{1}(\alpha;\gamma;z)=1+O(|\gamma|^{-1}).
\end{equation}
If $\gamma-\alpha$ and $z$ are bounded, $\gamma\rightarrow \infty$,
\begin{equation}
{}_{1}F_{1}(\alpha;\gamma;z)=e^{z}\left[1+O(|\gamma|^{-1})\right].
\end{equation}

\subsection{The generalized Laguerre polynomials} (see \cite{Bateman}, \cite{Lebedev} and \cite{Wang and Guo})
Laguerre's equation
\begin{equation}
z\frac{d^{2}y}{dz^{2}}+(\mu+1-z)\frac{dy}{dz}+ny=0 \ \ (n=0,1,2,\cdots)
\end{equation}
is a special case of the confluent hypergeometric equation
\begin{equation}
z\frac{d^{2}y}{dz^{2}}+(\mu+1-z)\frac{dy}{dz}-\alpha y=0
\end{equation}
by choosing $\alpha=-n=0,-1,-2,\cdots$. Then the polynomial
${}_{1}F_{1}(-n;\mu+1;z)$ is a solution of this equation for $\mu\not\in \mathbb{Z}_{<0}$. In particular, when $\mu=0$,
\begin{equation}
{}_{1}F_{1}(-n;1;z)=\frac{e^{z}}{n!}\cdot \frac{d^{n}}{dz^{n}}\left(e^{-z}z^{n}\right)
=\frac{1}{n!}\left(\frac{d}{dz}-1\right)^{n}z^{n},
\end{equation}
which is the Rodrigues formula. The polynomial $F(-n;1;z)$ is called the Laguerre polynomial.
\begin{definition}
The generalized Laguerre polynomials (The Sonine polynomials) are defined by
\begin{equation}
L_{n}^{\mu}(z)=\frac{z^{-\mu}e^{z}}{n!}\cdot \frac{d^{n}}{dz^{n}}\left(e^{-z}z^{n+\mu}\right),
\end{equation}
where $\mu\not\in \mathbb{Z}_{<0}$.
\end{definition}
Then
\begin{equation}
L_{n}^{\mu}(z)=\frac{z^{-\mu}}{n!}\cdot\left(\frac{d}{dz}-1\right)^{n} z^{n+\mu}=\frac{\Gamma(\mu+1+n)}{n!\cdot\Gamma(\mu+1)}{}_{1}F_{1}(-n;\mu+1;z).
\end{equation}
Obviously, $L_{n}^{\mu}(z)$ is a polynomial of degree $n$, and $L_{n}^{0}(z)={}_{1}F_{1}(-n;1;z)$.

The Laguerre's equation is equivalent to the statement that $L_{n}^{\mu}(z)$ is the eigenfunction with respect to eigenvalue $n$ of the second-order differential operator
\begin{equation}
\mathfrak{L}=-z\frac{d^{2}}{dz^{2}}+(z-\mu-1)\frac{d}{dz}.
\end{equation}
Then the operator is self-adjoint with respect to the inner product
\begin{equation}
\langle f,g\rangle=\int_{0}^{\infty}f(z)g(z)\omega(z)dz,
\end{equation}
where $\omega(z)$ is weight function, i.e.,
\begin{equation}
\langle \mathfrak{L}f, g\rangle=\langle f, \mathfrak{L}g\rangle.
\end{equation}

Based on the integral representations of confluent hypergeometric functions, one can get the degenerated cases for $\alpha\in\mathbb{Z}_{\leq 0}$, the basic one is
\begin{equation}
L_{n}^{\mu}(z)=\frac{(-1)^{n}}{2\pi i}\int^{(0+)}e^{zt}(1-t)^{\mu+n}t^{-n-1}dt,
\end{equation}
where the contour encircles $t=0$ in the positive sense, and $t=1$ is outside the contour, $|\arg(1-t)|<\pi$.

Set $t=1-\frac{v}{z}$ in the above equality, then
\begin{equation}
\begin{split}
L_{n}^{\mu}(z)&=e^{z}z^{-\mu}\frac{1}{2\pi i}\int^{(z+)}\frac{e^{-v}v^{\mu+n}}{(v-z)^{n+1}}dv\\
&=\frac{e^{z}z^{-\mu}}{n!}\cdot\frac{d^{n}}{dz^{n}}\left(z^{\mu+n}e^{-z}\right),
\end{split}
\end{equation}
which implies a differentiation formula
\begin{equation}
L_{n}^{m}(z)=(-1)^{m}\frac{d^{m}}{dz^{m}}L_{m+n}^{0}(z),
\end{equation}
where $n$ and $m$ are nonnegative integers.

From the integral representation, one can get
\begin{equation}
\sum_{n=0}^{\infty}L_{n}^{\mu}(z)t^{n}=\frac{e^{-\frac{zt}{1-t}}}{(1-t)^{\mu+1}}\ \ (|t|<1),
\end{equation}
the function of right hand side is called a generating function of $L_{n}^{\mu}(z)$.

Besides, if we consider a collection of generalized Laguerre polynomials $\{L_{n}^{\mu}(z)\}$, $n=0,1,2,\cdots$, then
\begin{equation}
\int_{0}^{\infty}z^{\mu}e^{-z}L_{n}^{\mu}(z)L_{n'}^{\mu}(z)dz=\frac{\Gamma(\mu+n+1)}{n!}\delta_{nn'},
\end{equation}
which is the orthogonality of the generalized Laguerre polynomials.

\subsection{Applications}
In statistics, ${}_{1}F_{1}(\alpha;\gamma;z)$ with integral and half integral values of the parameters $\alpha$ and $\gamma$, occurs in the distributions of many important statistics, such as F-statistics and $D^{2}$-statistics (see \cite{P. Nath}). Moreover, it is well-known that Error function and Incomplete function appear in statistics commonly, which can be represented by the confluent hypergeometric functions, i.e.,
\begin{equation}
\ erf x= \frac{2}{\sqrt{\pi}}\int_{0}^{x} e^{-t^{2}}dt=\frac{2x}{\sqrt{\pi}}e^{-x^{2}}{}_{1}F_{1}\left(1;\frac{3}{2};x^{2}\right),
\end{equation}
\begin{equation}
\Gamma_{x}(n)=\int_{0}^{x}e^{-t}t^{n-1}dt=\frac{1}{n}e^{-x}x^{n}{}_{1}F_{1}(1;n+1;x).
\end{equation}
The normal distribution function with mean $m$ and standard deviation $\sigma$ is given by
\begin{equation}
\begin{split}
&\frac{1}{\sigma\sqrt{2\pi}}\int_{-\infty}^{x}e^{-(t-m)^{2}/\left(2\sigma^{2}\right)}dt\\
&=\frac{1}{2}\ erfc \left(\frac{m-x}{\sigma\sqrt{2}}\right)\\
&=\frac{1}{2}\left[1-\ erf \left(\frac{m-x}{\sigma\sqrt{2}}\right)\right]\\
&=\frac{1}{2}\left[1-\frac{2\left(\frac{m-x}{\sigma\sqrt{2}}\right)}
{\sqrt{\pi}}e^{-\left(\frac{m-x}{\sigma\sqrt{2}}\right)^{2}}
{}_{1}F_{1}\left(1;\frac{3}{2};\left(\frac{m-x}{\sigma\sqrt{2}}\right)^{2}\right)\right].
\end{split}
\end{equation}
The Poisson-Charlier polynomials, arising in the calculus of probability, can be expressed by means of generalized Laguerre polynomials as
\begin{equation}
C_{n}(x;\mu)=_{2}F_{0}\left(-n,-x;-\frac{1}{\mu}\right)=(-1)^{n}\cdot n!\cdot L_{n}^{-1-x}\left(-\frac{1}{\mu}\right).
\end{equation}

In quantum mechanics, the study of almost all physical systems allowing exact solutions for Schr$\ddot{o}$dinger equation (e.g., the harmonic oscillator, the hydrogen atom, the Morse, P$\ddot{o}$schl-Teller, Wood-Saxon, Hulth$\acute{e}$n or Eckart potentials) is reduced to the analysis of either hypergeometric or confluent hypergeometric equations (see \cite{J. Negro} and \cite{J. B. Seaborn}). Besides, for the two-dimensional Coulomb potential, each physical normalized eigenfunction associated to a bounded state can be described in two ways by means of the confluent hypergeometric functions.

What is more, the confluent hypergeometric equations are related to some equations owning one regular singularity and one irregular singularity, such as the Whittaker equation, the Bessel equation and the Coulomb wave equation. Specifically, the Bessel equation
\begin{equation}
\frac{d^{2}f}{dz^{2}}+\frac{1}{z}\frac{df}{dz}+\left(1-\frac{\nu^{2}}{z^{2}}\right)f(z)=0
\end{equation}
can be viewed as a special case of the confluent hypergeometric equation under the transformation
\begin{equation}
f(z)=z^{\nu}e^{-\frac{\xi}{2}}y(z),\ \xi=2iz,
\end{equation}
where $y(\xi)$ satisfies the confluent hypergeometric equation
\begin{equation}
\xi\frac{d^{2}y}{d\xi^{2}}+(2\nu+1-\xi)\frac{dy}{d\xi}-\left(\nu+\frac{1}{2}\right)y(\xi)=0
\end{equation}
with $\alpha=\nu+\frac{1}{2}$, $\gamma=2\nu+1=2\alpha$ in (\ref{E:4}).

The Whittaker equation
\begin{equation}
\frac{d^{2}\omega}{dz^{2}}+\left(-\frac{1}{4}+\frac{k}{z}+\frac{\frac{1}{4}-m^{2}}{z^{2}}\right)\omega=0,
\end{equation}
which is the normal form of confluent hypergeometric equation (\ref{E:4}) via the substitutions $y(z)=e^{\frac{z}{2}}z^{-\frac{\gamma}{2}}\omega(z)$, $k=\frac{\gamma}{2}-\alpha$ and $m=\frac{\gamma-1}{2}$. It has a regular singularity at the origin with exponents $\frac{1}{2}\pm m$, and an irregular singularity at infinity of rank one. If $2m$ is not an integer, then the two linearly independent solutions at origin are
\begin{equation}
M_{k,m}(z)=e^{-\frac{z}{2}}z^{\frac{\gamma}{2}}F(\alpha;\gamma;z)
=e^{-\frac{z}{2}}z^{\frac{1}{2}+m}{}_{1}F_{1}\left(\frac{1}{2}+m-k;1+2m;z\right)
\end{equation}
and
\begin{equation}
M_{k,-m}(z)=e^{-\frac{z}{2}}z^{1-\frac{\gamma}{2}}F(\alpha-\gamma+1;2-\gamma;z)
=e^{-\frac{z}{2}}z^{\frac{1}{2}-m}{}_{1}F_{1}\left(\frac{1}{2}-m-k;1-2m;z\right).
\end{equation}

Moreover, the confluent hypergeometric functions are connected with representations of the group of third-order triangular matrices (see \cite{N. Ja. Vilenkin}). The elements of this group are of the form
\begin{equation}
\left(
\begin{array}{ccc}
1 & \alpha & \beta\\
0 & \gamma & \delta\\
0 & 0 & 1\\
\end{array}
\right),
\end{equation}
where $\alpha$, $\beta$, $\gamma$, $\delta$ are real numbers, and $\gamma>0$. Vilenkin (see \cite{N. Ja. Vilenkin}) constructs irreducible representations of this group, in which the diagonal matrices correspond to operators of multiplication by an exponential function. The other group elements correspond to integral operators whose kernels can be expressed in terms of Whittaker functions. The identification can be used to obtain various properties of the Whittaker functions, including recurrence relations and derivatives.

In this paper, we will use Nevanlinna's value distribution theory and Wiman-Valiron theory to describe some new properties of confluent hypergeometric functions. For the convenience of readers, we list
the following standard notations and results in Nevanlinna theory (see \cite{Hayman}).

Let $f$ be a nonconstant meromorphic function in the complex plane. Then the proximity function $m(r,f)$, counting function $N(r,f)$, reduced counting function $\overline{N}(r,f)$, characteristic function $T(r,f)$ and order $\sigma(f)$ are defined by
\begin{equation}
m(r,f)=\dfrac{1}{2\pi}\int_{0}^{2\pi}\log^{+}|f(re^{i\theta})|d\theta,
\end{equation}
\begin{equation}
N(r,f)=\int_{0}^{r}\dfrac{n(t,f)-n(0,f)}{t}dt+n(0,f)\log r,
\end{equation}
\begin{equation}
\overline{N}(r,f)=\int_{0}^{r}\dfrac{\overline{n}(t,f)-\overline{n}(0,f)}{t}dt +\overline{n}(0,f)\log r,
\end{equation}
\begin{equation}
T(r,f)=m(r,f)+N(r,f),
\end{equation}
\begin{equation}
\sigma(f)=\limsup_{r\rightarrow\infty}\frac{\log^{+} T(r,f)}{\log r},
\end{equation}
respectively, where $\log^{+}x=\max\{\log x,0\}$ for all $x\geq0$, $n(t,f)$ denotes the number of poles of $f$ in the closed disc $\overline{D(0,t)}=\{z:|z|\leq t\}$, counting multiplicities and $\overline{n}(t,f)$ denotes the number of poles of $f$ in $\overline{D(0,t)}$, ignoring multiplicities.

We recall the following results:

(i) The arithmetic properties of $T(r,f)$ and $m(r,f)$:
\begin{equation}
T(r,fg)\leq T(r,f)+T(r,g),
\end{equation}
\begin{equation}
T(r,f+g)\leq T(r,f)+T(r,g)+O(1).
\end{equation}
The same inequalities hold for $m(r,f)$.

(ii) $T(r,f)$ is an increasing function with respect to $r$. Moreover, $f$ is a rational function if and only if
\begin{equation}
T(r,f)=O(\log r).
\end{equation}

(iii)Nevanlinna's First Main Theorem:
\begin{equation}
T(r,f)=T\left(r,\dfrac{1}{f}\right)+O(1).
\end{equation}

(iv) The Logarithmic Derivative Lemma: If $f$ is of finite order, then
\begin{equation}
m\left(r,\dfrac{f'}{f}\right)=O(\log r).
\end{equation}
If $f$ is of infinite order, then
\begin{equation}
m\left(r,\frac{f'}{f}\right)=O(\log rT(r,f)),
\end{equation}
outside of a possible exceptional set of finite linear measure.

(v) Nevanlinna's Second Main Theorem:
\begin{equation}
(q-2)T(r,f)\leq \sum_{j=1}^{q}N\left(r,\frac{1}{f-a_j}\right)+S(r,f),
\end{equation}
where $a_1,a_2,\cdots,a_q$ are distinct complex numbers in $\mathbb{\hat{C}}=\mathbb{C}\bigcup\{\infty\}$ and $S(r,f)$ denotes a quantity satisfying $S(r,f)=O(\log(rT(r,f)))$,
outside of a possible exceptional set of finite linear measure. If $f$ is of finite order, then $S(r,f)=O(\log r)$.

\section{Some Lemmas}
In order to prove our main results in Section 3-5, we need the following lemmas.

\begin{lemma}[\cite{Laine}]
\label{L:1}
If $f$ is an entire function of order $\sigma$, then
\begin{equation}
\sigma=\lim\sup_{r\rightarrow\infty}\frac{\log^{+}\nu_{f}(r)}{\log r}
=\lim\sup_{r\rightarrow\infty}\frac{\log^{+}\log^{+}\mu_{f}(r)}{\log r},
\end{equation}
where $\mu_{f}(r)$ and $\nu_{f}(r)$ are the maximum term and central index of $f(z)$ respectively.

\end{lemma}

\begin{lemma}[\cite{Laine}]
\label{L:2}
Let $f$ be a transcendental entire function, let $0<\delta<\frac{1}{4}$ and $z$ be such that $|z|=r$ and that
\begin{equation}
|f(z)|>M(r,f)\nu_{f}(r)^{-\frac{1}{4}+\delta}
\end{equation}
holds. Then there exists a set $F\subset \mathbb{R}_{+}$ of finite logarithmic measure, i.e., $\int_{F}\frac{dt}{t}<+\infty$, such that
\begin{equation}
f^{(m)}(z)=\left(\frac{\nu_{f}(r)}{z}\right)^{m}(1+o(1))f(z)
\end{equation}
holds for all $m\geq 0$ and all $r\not\in F$.
\end{lemma}

\begin{lemma}[\cite{Hayman}]
\label{L:3}
Let $f(z)$ be a nonconstant entire function and $M(r, f):=\max_{|z|=r}|f(z)|$, then
\begin{equation}
T(r, f)\leq \log^{+}M(r,f)\leq \frac{R+r}{R-r}T(R, f)
\end{equation}
for $0\leq r < R < \infty$.
\end{lemma}

\begin{lemma}[\cite{Hayman}]
\label{L:4}
Let $h(z)$ be a nonconstant entire function and $\displaystyle f(z)=e^{h(z)}$, $\sigma$ and $\mu$ be the order and the lower order of $f(z)$.

\par(i) If $h(z)$ is a polynomial of degree $p$, then $\sigma=\mu=p$.

\par(ii) If $h(z)$ is a transcendental entire function, then $\sigma=\mu=\infty$.
\end{lemma}

\section{Nevanlinna Characteristic Of ${}_{1}F_{1}(\alpha,\gamma;z)$}

As we know, around the origin, ${}_{1}F_{1}(\alpha;\gamma;z)$ is one solution of $z\frac{d^{2}y}{dz^{2}}+(\gamma-z)\frac{dy}{dz}-\alpha y(z)=0$, which is an entire function. In this section,
we shall discuss the rate of growth of the function by showing estimates of maximum modulus function $M(r,{}_{1}F_{1}(\alpha;\gamma;z))$ and Nevanlinna characteristic function $T(r,{}_{1}F_{1}(\alpha;\gamma;z))$, computing the order $\sigma({}_{1}F_{1}(\alpha;\gamma;z))$, and providing a better estimation of $m\left(r,\frac{{}_{1}F_{1}'(\alpha;\gamma;z)}{{}_{1}F_{1}(\alpha;\gamma;z)}\right)$, which are our main results.

\begin{theorem}
\label{T:1}
Let ${}_{1}F_{1}(\alpha;\gamma;z)$ be a confluent hypergeometric function,
where $\gamma\neq 0,-1, -2, \cdots$. Then $T(r,{}_{1}F_{1}(\alpha;\gamma;z))=O(r)$ for $\alpha\not\in \mathbb{Z}_{\leq 0}$, $T\left(r,{}_{1}F_{1}(\alpha;\gamma;z)\right)=O(\log r)$ for $\alpha\in \mathbb{Z}_{\leq 0}$.
\end{theorem}
\begin{proof}
Since the confluent hypergeometric function
\begin{equation}
{}_{1}F_{1}(\alpha;\gamma;z)=\sum_{n=0}^{\infty}\frac{(\alpha)_{n}}{n!(\gamma)_{n}}z^{n}\ (\gamma\neq 0,-1,-2,\cdots),
\end{equation}
then
\begin{equation}
\begin{split}
&m\left(r,{}_{1}F_{1}(\alpha;\gamma;z)\right)\\
&=\frac{1}{2\pi}\int_{0}^{2\pi}\log^{+}|{}_{1}F_{1}(\alpha,\gamma;re^{i\theta})|d\theta\\
&=\frac{1}{2\pi}\int_{0}^{2\pi}\log^{+}\left|\sum_{n=0}^{\infty}\frac{(\alpha)_{n}}{n!(\gamma)_{n}}\left(re^{i\theta}\right)^{n}\right|d\theta\\
&=\frac{1}{2\pi}\int_{0}^{2\pi}\log^{+}\left|\sum_{n=0}^{\infty}\frac{(\alpha)_{n}}{n!(\gamma)_{n}}r^{n}e^{in\theta}\right|d\theta\\
&\leq \frac{1}{2\pi}\int_{0}^{2\pi}\log^{+}\sum_{n=0}^{\infty}\left|\frac{(\alpha)_{n}}{n!(\gamma)_{n}}r^{n}e^{in\theta}\right|d\theta\\
&=\frac{1}{2\pi}\int_{0}^{2\pi}\log^{+}\sum_{n=0}^{\infty}\left|\frac{(\alpha)_{n}}{n!(\gamma)_{n}}\right|r^{n}d\theta\\
&=\log^{+}\sum_{n=0}^{\infty}\left|\frac{(\alpha)_{n}}{n!(\gamma)_{n}}\right|r^{n}.
\end{split}
\end{equation}
Next, we will distinguish three cases to discuss for $\alpha\not\in \mathbb{Z}_{\leq 0}$.

{\noindent\bf Case 1.} If $\Re \alpha< \Re \gamma$, then we can find a smallest $j\in \mathbb{Z}^{+}$ such that $\Re \alpha+j>0$ and $|\alpha+j|<|\gamma+j|$. Denote
\begin{equation}
C(\alpha,\gamma)=1+\max_{1\leq i\leq j}\left|\frac{(\alpha)_{i}}{(\gamma)_{i}}\right|.
\end{equation}
Thus,
\begin{equation}
\left|\frac{(\alpha)_{n}}{(\gamma)_{n}}\right|\leq C(\alpha,\gamma)
\end{equation}
whenever $0\leq n\leq j$.

Since $\left|\frac{\alpha+k}{\gamma+k}\right|<1$ for $k\geq j$,
we have
\begin{equation}
\left|\frac{(\alpha)_{n}}{(\gamma)_{n}}\right|
=\left|\frac{(\alpha)_{j}}{(\gamma)_{j}}\cdot\frac{(\alpha+j)_{n-j}}{(\gamma+j)_{n-j}}\right|
\leq C(\alpha,\gamma)\cdot \left|\frac{(\alpha+j)_{n-j}}{(\gamma+j)_{n-j}}\right|< C(\alpha,\gamma)
\end{equation}
whenever $j<n$. So
\begin{equation}
\left|\frac{(\alpha)_{n}}{(\gamma)_{n}}\right|\leq C(\alpha,\gamma)
\end{equation}
for $n\geq 0$. It implies that
\begin{equation}
m\left(r,{}_{1}F_{1}(\alpha;\gamma;z)\right)\leq \log^{+} \left(C(\alpha,\gamma)\sum_{n=0}^{\infty}\frac{r^{n}}{n!}\right)
=\log^{+} \left(C(\alpha,\gamma) e^{r}\right)=O(r).
\end{equation}
{\noindent\bf Case 2.} If $\Re \alpha> \Re \gamma$, then we can find a smallest $\widetilde{j}\in \mathbb{Z}^{+}$ such that $\Re \gamma+\widetilde{j}>0$ and $|\alpha+\widetilde{j}|>|\gamma+\widetilde{j}|$. Thus, there exists a positive integer $\beta\geq 2$ such that
$\beta|\gamma+\widetilde{j}|>|\alpha+\widetilde{j}|$. Similarly, we denote
\begin{equation}
\widetilde{C}(\alpha,\gamma)=1+\max_{1\leq i\leq \widetilde{j}}\left|\frac{(\alpha)_{i}}{(\gamma)_{i}}\right|.
\end{equation}
Then
\begin{equation}
\left|\frac{(\alpha)_{n}}{(\gamma)_{n}}\right|\leq \widetilde{C}(\alpha,\gamma)
\end{equation}
whenever $0\leq n\leq \widetilde{j}$.

Note that $\left|\frac{\alpha+k}{\gamma+k}\right|>1$ for $k\geq \widetilde{j}$
and $\left|\frac{\alpha+k}{\gamma+k}\right|$ is decreasing with respect to $k$, owning lower bound $1$.
So we can obtain that
\begin{equation}
\begin{split}
&\left|\frac{(\alpha)_{n}}{(\gamma)_{n}}\right|=\left|\frac{(\alpha)_{\widetilde{j}}}{(\gamma)_{\widetilde{j}}}
\cdot\frac{(\alpha+\widetilde{j})_{n-\widetilde{j}}}{(\gamma+\widetilde{j})_{n-\widetilde{j}}}\right|\\
&\leq \widetilde{C}(\alpha,\gamma)\cdot \left|\frac{(\alpha+\widetilde{j})_{n-\widetilde{j}}}{(\gamma+\widetilde{j})_{n-\widetilde{j}}}\right|\\
&= \widetilde{C}(\alpha,\gamma)\cdot \left|\frac{(\alpha+\widetilde{j})_{n-\widetilde{j}}}{\beta^{n}\cdot(\gamma+\widetilde{j})_{n-\widetilde{j}}}\right|\cdot\beta^{n}\\
&\leq \widetilde{C}(\alpha,\gamma)\cdot\beta^{n}
\end{split}
\end{equation}
whenever $\widetilde{j}<n$. Then
\begin{equation}
\left|\frac{(\alpha)_{n}}{(\gamma)_{n}}\right|\leq \widetilde{C}(\alpha,\gamma)\cdot\beta^{n}
\end{equation}
for $n\geq 0$. Hence,
\begin{equation}
m\left(r,{}_{1}F_{1}(\alpha;\gamma;z)\right)\leq \log^{+} \left(\widetilde{C}(\alpha,\gamma)\sum_{n=0}^{\infty}\frac{(\beta r)^{n}}{n!}\right)
=\log^{+} \left(\widetilde{C}(\alpha,\gamma) e^{\beta r}\right)=O(r).
\end{equation}

{\noindent\bf Case 3.} If $\Re \alpha= \Re \gamma$, then there exists a smallest $\widehat{j}\in \mathbb{Z}^{+}$ such that $\Re \alpha+\widehat{j}=\Re \gamma+\widehat{j}>0$. Next, we can follow Case 1 and Case 2 for $|\alpha+\widehat{j}|\leq|\gamma+\widehat{j}|$ and $|\alpha+\widehat{j}|>|\gamma+\widehat{j}|$ respectively.

Since the confluent hypergeometric function ${}_{1}F_{1}(\alpha;\gamma;z)$ is an entire function, combining the above three cases, we get
\begin{equation}
T\left(r,{}_{1}F_{1}(\alpha;\gamma;z)\right)=O(r).
\end{equation}

Obviously, when $\alpha=-n=0,-1,-2,\cdots$, ${}_{1}F_{1}(-n;\gamma;z)$ is a polynomial of degree $n$, then
\begin{equation}
T(r,{}_{1}F_{1}(-n;\gamma;z))=n\log r+O(1).
\end{equation}

This completes the proof.
\end{proof}

Theorem \ref{T:1} shows two different growth levels for the characteristic functions of degenerated and non-degenerated confluent hypergeometric functions. From Lemma \ref{L:3} and the proof of Theorem  \ref{T:1}, we can get a similar situation occurs in their corresponding maximum modulus functions $M(r, {}_{1}F_{1}(\alpha;\gamma;z))$ immediately.

\begin{theorem}
\label{T:1'}
Let ${}_{1}F_{1}(\alpha;\gamma;z)$ be a confluent hypergeometric function,
where $\gamma\neq 0,-1, -2, \cdots$. Then $M(r,{}_{1}F_{1}(\alpha;\gamma;z))=O(e^r)$ for $\alpha\not\in \mathbb{Z}_{\leq 0}$, $M\left(r,{}_{1}F_{1}(\alpha;\gamma;z)\right)=O(r^{-\alpha})$ for $\alpha\in \mathbb{Z}_{\leq 0}$.
\end{theorem}

The following theorem shows that, making use of the notion of order, these two growth levels can be distinguished by two different numbers simply.

\begin{theorem}
\label{T:2}
Let ${}_{1}F_{1}(\alpha;\gamma;z)$ be a confluent hypergeometric function,
where $\gamma\neq 0,-1, -2, \cdots$. Then $\sigma\left({}_{1}F_{1}(\alpha;\gamma;z)\right)=1$ for $\alpha\not\in \mathbb{Z}_{\leq 0}$,
$\sigma\left({}_{1}F_{1}(\alpha;\gamma;z)\right)=0$ for $\alpha\in \mathbb{Z}_{\leq 0}$, where $\sigma\left({}_{1}F_{1}(\alpha;\gamma;z)\right)$ is the order of ${}_{1}F_{1}(\alpha;\gamma;z)$.
\end{theorem}
\begin{proof}
From the proof of Theorem \ref{T:1}, by the definition of order, we can get $\sigma\left({}_{1}F_{1}(\alpha;\gamma;z)\right)\leq1$ for $\alpha\not\in \mathbb{Z}_{\leq 0}$.

In order to verify $\sigma\left({}_{1}F_{1}(\alpha;\gamma;z)\right)\geq 1$, we will consider three different cases for $\alpha\not\in \mathbb{Z}_{\leq 0}$.

{\noindent\bf Case 1.} If $\Re \alpha< \Re \gamma$, then we can find a smallest $j\in \mathbb{Z}^{+}$ such that $\Re \alpha+j>0$ and $|\alpha+j|<|\gamma+j|$. Denote
\begin{equation}
D(\alpha,\gamma)=\min_{1\leq i\leq j}\left|\frac{(\alpha)_{i}}{(\gamma)_{i}}\right|.
\end{equation}
Thus,
\begin{equation}
\left|\frac{(\alpha)_{n}}{(\gamma)_{n}}\right|\geq D(\alpha,\gamma)>0
\end{equation}
whenever $1\leq n\leq j$.

Note that $\left|\frac{\alpha+k}{\gamma+k}\right|<1$ for $k\geq j$
and $\left|\frac{\alpha+k}{\gamma+k}\right|$ is increasing with respect to $k$, having upper bound $1$.
We define
\begin{equation}
\delta:=\left|\frac{\alpha+j}{\gamma+j}\right|,
\end{equation}
then
\begin{equation}
\left|\frac{(\alpha)_{n}}{(\gamma)_{n}}\right|
=\left|\frac{(\alpha)_{j}}{(\gamma)_{j}}\cdot\frac{(\alpha+j)_{n-j}}{(\gamma+j)_{n-j}}\right|
\geq D(\alpha,\gamma)\cdot \left|\frac{(\alpha+j)_{n-j}}{(\gamma+j)_{n-j}}\right|
\geq D(\alpha,\gamma)\cdot \delta^{n-j}\geq D(\alpha,\gamma)\cdot \delta^{n}
\end{equation}
whenever $j<n$. Hence,
\begin{equation}
\left|\frac{(\alpha)_{n}}{(\gamma)_{n}}\right|\geq D(\alpha,\gamma)\cdot \delta^{n}
\end{equation}
for $n\geq 1$. It implies that
\begin{equation}
\left|\frac{(\alpha)_{n}}{n!(\gamma)_{n}}\right|r^{n}\geq D(\alpha,\gamma)\cdot \frac{(\delta r)^{n}}{n!}
\end{equation}
for $n\geq 1$.

Set $f(z)=D(\alpha,\gamma)e^{\delta z}-D(\alpha,\gamma)+1=1+D(\alpha,\gamma)\cdot\sum_{n=1}^{\infty}\frac{(\delta z)^{n}}{n!}$, there exists $r_{0}>0$ such that
\begin{equation}
D(\alpha,\gamma)\cdot\delta r>1
\end{equation}
whenever $r>r_{0}$, then
\begin{equation}
\begin{split}
&\mu_{f(z)}(r)\\
&=\max_{n\geq 1}\left\{D(\alpha,\gamma)\cdot\frac{\delta^{n}}{n!}\cdot r^{n}\right\}\\
&=D(\alpha,\gamma)\cdot\frac{\delta^{\nu_{f(z)}(r)}}{\nu_{f(z)}(r)!}\cdot r^{\nu_{f(z)}(r)}\\
&\leq \left|\frac{(\alpha)_{\nu_{f(z)}(r)}}{\nu_{f(z)}(r)!(\gamma)_{\nu_{f(z)}(r)}}\right|r^{\nu_{f(z)}(r)}\\
&\leq \max_{n\geq 0}\left\{\left|\frac{(\alpha)_{n}}{n!(\gamma)_{n}}\right|r^{n}\right\}\\
&=\mu_{{}_{1}F_{1}(\alpha;\gamma;z)}(r)
\end{split}
\end{equation}
whenever $r>r_{0}$, where $\nu_{f(z)}(r)=\max\left\{m,\mu_{f(z)}(r)=D(\alpha,\gamma)\cdot\frac{\delta^{m}}{m!}\cdot r^{m}\right\}$. Hence, by Lemma \ref{L:1} and Lemma \ref{L:4}, we have
\begin{equation}
\begin{split}
&\sigma({}_{1}F_{1}(\alpha;\gamma;z))\\
&=\lim\sup_{r\rightarrow\infty}\frac{\log^{+}\log^{+}\mu_{{}_{1}F_{1}(\alpha;\gamma;z)}(r)}{\log r}\\
&\geq \lim\sup_{r\rightarrow\infty}\frac{\log^{+}\log^{+}\mu_{f(z)}(r)}{\log r}\\
&=\sigma(f(z))\\
&=1.
\end{split}
\end{equation}
{\noindent\bf Case 2.} If $\Re \alpha> \Re \gamma$, then we can find a smallest $\widetilde{j}\in \mathbb{Z}^{+}$ such that $\Re \gamma+\widetilde{j}>0$ and $|\alpha+\widetilde{j}|>|\gamma+\widetilde{j}|$. Similarly, we denote
\begin{equation}
\widetilde{D}(\alpha,\gamma)=\min_{1\leq i\leq \widetilde{j}}\left|\frac{(\alpha)_{i}}{(\gamma)_{i}}\right|.
\end{equation}
Then
\begin{equation}
\left|\frac{(\alpha)_{n}}{(\gamma)_{n}}\right|\geq \widetilde{D}(\alpha,\gamma)
\end{equation}
whenever $1\leq n\leq \widetilde{j}$.

Since $\left|\frac{\alpha+k}{\gamma+k}\right|>1$ for $k\geq \widetilde{j}$, we have
\begin{equation}
\left|\frac{(\alpha)_{n}}{(\gamma)_{n}}\right|=\left|\frac{(\alpha)_{\widetilde{j}}}{(\gamma)_{\widetilde{j}}}
\cdot\frac{(\alpha+\widetilde{j})_{n-\widetilde{j}}}{(\gamma+\widetilde{j})_{n-\widetilde{j}}}\right|
\geq \widetilde{D}(\alpha,\gamma)\cdot\left|\frac{(\alpha+\widetilde{j})_{n-\widetilde{j}}}{(\gamma+\widetilde{j})_{n-\widetilde{j}}}\right|
>\widetilde{D}(\alpha,\gamma)
\end{equation}
whenever $\widetilde{j}<n$. So
\begin{equation}
\left|\frac{(\alpha)_{n}}{(\gamma)_{n}}\right|\geq \widetilde{D}(\alpha,\gamma)
\end{equation}
for $n\geq 1$. It implies that
\begin{equation}
\left|\frac{(\alpha)_{n}}{n!(\gamma)_{n}}\right|r^{n}\geq \frac{\widetilde{D}(\alpha,\gamma)}{n!}\cdot r^{n}
\end{equation}
for $n\geq 1$.

Set $g(z)=\widetilde{D}(\alpha,\gamma)e^{z}-\widetilde{D}(\alpha,\gamma)+1=1+\widetilde{D}(\alpha,\gamma)\cdot\sum_{n=1}^{\infty}
\frac{z^{n}}{n!}$, there exists $\widetilde{r}_{0}>0$ such that
\begin{equation}
\widetilde{D}(\alpha,\gamma)\cdot r>1
\end{equation}
whenever $r>\widetilde{r}_{0}$, then
\begin{equation}
\begin{split}
&\mu_{g(z)}(r)\\
&=\max_{n\geq 1}\left\{\frac{\widetilde{D}(\alpha,\gamma)}{n!}\cdot r^{n}\right\}\\
&=\widetilde{D}(\alpha,\gamma)\cdot\frac{\delta^{\nu_{g(z)}(r)}}{\nu_{g(z)}(r)!}\cdot r^{\nu_{g(z)}(r)}\\
&\leq \left|\frac{(\alpha)_{\nu_{g(z)}(r)}}{\nu_{g(z)}(r)!(\gamma)_{\nu_{g(z)}(r)}}\right|r^{\nu_{g(z)}(r)}\\
&\leq \max_{n\geq 0}\left\{\left|\frac{(\alpha)_{n}}{n!(\gamma)_{n}}\right|r^{n}\right\}\\
&=\mu_{{}_{1}F_{1}(\alpha;\gamma;z)}(r)
\end{split}
\end{equation}
whenever $r>\widetilde{r}_{0}$, where $\nu_{g(z)}(r)=\max\left\{m,\mu_{g(z)}(r)=\frac{\widetilde{D}(\alpha,\gamma)}{m!}\cdot r^{m}\right\}$. Hence,
by Lemma \ref{L:1} and Lemma \ref{L:4}, we have
\begin{equation}
\begin{split}
&\sigma({}_{1}F_{1}(\alpha;\gamma;z))\\
&=\lim\sup_{r\rightarrow\infty}\frac{\log^{+}\log^{+}\mu_{{}_{1}F_{1}(\alpha,\gamma;z)}(r)}{\log r}\\
&\geq \lim\sup_{r\rightarrow\infty}\frac{\log^{+}\log^{+}\mu_{g(z)}(r)}{\log r}\\
&=\sigma(g(z))\\
&=1.
\end{split}
\end{equation}
{\noindent\bf Case 3.} If $\Re \alpha= \Re \gamma$, then there exists a smallest $\widehat{j}\in \mathbb{Z}^{+}$ such that $\Re \alpha+\widehat{j}=\Re \gamma+\widehat{j}>0$. Next, we can follow Case 1 and Case 2 for $|\alpha+\widehat{j}|\leq|\gamma+\widehat{j}|$ and $|\alpha+\widehat{j}|>|\gamma+\widehat{j}|$ respectively.

It is easy to see that $\sigma({}_{1}F_{1}(\alpha;\gamma;z))=0$ when $\alpha\in \mathbb{Z}_{\leq 0}$.

This completes the proof.
\end{proof}
\medskip

Let $f$ be a meromorphic function, it is well-known that the following logarithmic derivative estimate
\begin{equation}
m\left(r, \frac{f'(z)}{f(z)}\right)=O(\log rT(r, f))=S(r, f),
\end{equation}
holds outside a possible set of finite linear measure, where the notation $S(r,f)$ means that the expression is of
$o(T(r,f))$. It shows that the proximity function of the logarithmic derivative of $f(z)$ grows much slower than the characteristic function of $f(z)$. By Theorem \ref{T:2}, we have known that ${}_{1}F_{1}(\alpha;\gamma;z)$ is of finite order,
then the growth of $m\left(r, \frac{{}_{1}F_{1}'(\alpha;\gamma;z)}{{}_{1}F_{1}(\alpha;\gamma;z)}\right)$ is not exceeded $O(\log r)$. Next, we will provide a more precise estimation for $m\left(r, \frac{{}_{1}F_{1}'(\alpha;\gamma;z)}{{}_{1}F_{1}(\alpha;\gamma;z)}\right)$.

\begin{theorem}
\label{T:3}
Let ${}_{1}F_{1}(\alpha;\gamma;z)$ be a confluent hypergeometric function,
where $\gamma\neq 0,-1, -2, \cdots$. Then $m\left(r,\frac{{}_{1}F_{1}'(\alpha;\gamma;z)}{{}_{1}F_{1}(\alpha;\gamma;z)}\right)=O(1)$, outside a possible set of finite logarithmic measure.
\end{theorem}

\begin{proof}
Note that ${}_{1}F_{1}(\alpha;\gamma;z)$ is an entire function, by Lemma \ref{L:2}, we have
\begin{equation}
\frac{{}_{1}F_{1}''(\alpha;\gamma;z)}{{}_{1}F_{1}(\alpha;\gamma;z)}=\left(\frac{\nu_{{}_{1}F_{1}(\alpha;\gamma;z)}(r)}{z}\right)^{2}(1+o(1))
\end{equation}
and
\begin{equation}
\frac{{}_{1}F_{1}'(\alpha;\gamma;z)}{{}_{1}F_{1}(\alpha;\gamma;z)}=\frac{\nu_{{}_{1}F_{1}(\alpha;\gamma;z)}(r)}{z}(1+o(1))
\end{equation}
holding for $r=|z|\not\in E$, where $E\subset \mathbb{R}_{+}$ is a set of finite logarithmic measure, and $z$ is chosen such that $|{}_{1}F_{1}(\alpha;\gamma;z)|=M(r,{}_{1}F_{1}(\alpha;\gamma;z)):=\max_{|z|=r}\left|{}_{1}F_{1}(\alpha;\gamma;z)\right|$.

Since ${}_{1}F_{1}(\alpha;\gamma;z)$ is a solution of (\ref{E:4}) around the origin, we can deduce
\begin{equation}
\left(\nu_{{}_{1}F_{1}(\alpha;\gamma;z)}(r)\right)^{2}(1+o(1))+(\gamma-z)\nu_{{}_{1}F_{1}(\alpha;\gamma;z)}(r)(1+o(1))-\alpha z=0.
\end{equation}

Then
\begin{equation}
\nu_{{}_{1}F_{1}(\alpha;\gamma;z)}(r)\sim\left|\frac{-(\gamma-z)\pm\sqrt{(\gamma-z)^{2}+4\alpha}}{2}\right|\leq ar,
\end{equation}
$r\not\in E$, where $a$ is a positive real constant.

It implies that
\begin{equation}
m\left(r,\frac{{}_{1}F_{1}'(\alpha;\gamma;z)}{{}_{1}F_{1}(\alpha;\gamma;z)}\right)=O(1),
\end{equation}
outside a possible set of finite logarithmic measure.

This completes the proof.
\end{proof}

\medskip
\section{Zeros distribution of ${}_{1}F_{1}(\alpha;\gamma;z)$}
Many problems in mathematical physics can be solved with the help of the distribution of zeros of confluent hypergeometric functions. If $\alpha\not\in \mathbb{Z}_{\leq 0}$ and $\gamma-\alpha\not\in \mathbb{Z}_{\leq 0}$, then ${}_{1}F_{1}(\alpha;\gamma;z)$ has infinitely many zeros in $\mathbb{C}$. Naturally, it is desirable to explore the distribution of these zeros.
These are many known results about it for real parameters (see \cite{Slater}).

When $\alpha$, $\gamma$, $z\in \mathbb{R}$, since
\begin{equation}
{}_{1}F_{1}(\alpha;\gamma;z)=\frac{\Gamma(\gamma)}{\Gamma(\alpha)}e^{z}z^{\alpha-\gamma}\left[1+O(|z|^{-1})\right]\neq 0
\end{equation}
as $z\rightarrow +\infty$, and
\begin{equation}
{}_{1}F_{1}(\alpha;\gamma;z)=\frac{\Gamma(\gamma)}{\Gamma(\gamma-\alpha)}(-z)^{-\alpha}\left[1+O(|z|^{-1})\right]\neq 0
\end{equation}
as $z\rightarrow -\infty$, then ${}_{1}F_{1}(\alpha;\gamma;z)$ has finitely many real zeros. Specifically, let $n^{+}(\alpha,\gamma)$ be the number of positive zeros, then
\begin{equation}
n^{+}(\alpha,\gamma)=\lceil-\alpha\rceil
\end{equation}
when $\alpha<0$ and $\gamma\geq 0$,
\begin{equation}
n^{+}(\alpha,\gamma)=0
\end{equation}
when $\alpha\geq 0$ and $\gamma\geq 0$,
\begin{equation}
n^{+}(\alpha,\gamma)=1
\end{equation}
when $\alpha\geq 0$ and $-1< \gamma <0$,
\begin{equation}
n^{+}(\alpha,\gamma)=\left\lfloor-\frac{\gamma}{2}\right\rfloor-\left\lfloor-\frac{\gamma+1}{2}\right\rfloor
\end{equation}
when $\alpha\geq 0$ and $\gamma\leq 1$,
\begin{equation}
n^{+}(\alpha,\gamma)=\left\lceil-\alpha\right\rceil-\left\lceil-\gamma\right\rceil
\end{equation}
when $\lceil-\alpha\rceil\geq\lceil-\gamma\rceil$, $\alpha< 0$ and $\gamma< 0$,
\begin{equation}
n^{+}(\alpha,\gamma)=\left\lfloor\frac{1}{2}(\lceil-\gamma\rceil-\lceil-\alpha\rceil+1)\right\rfloor
-\left\lfloor\frac{1}{2}(\lceil-\gamma\rceil-\lceil-\alpha\rceil)\right\rfloor
\end{equation}
when $\lceil-\gamma\rceil>\lceil-\alpha\rceil>0$. Here, the notations $\lceil x\rceil$ and $\lfloor x \rfloor$ stand for the integer such that $x\leq \lceil x \rceil<x+1$, and that for which $x-1<\lfloor x \rfloor \leq x$ respectively, where $x$ is real. Besides, the number of negative zeros $n^{-}(\alpha,\gamma)$ is given by
\begin{equation}
n^{-}(\alpha,\gamma)=n^{+}(\gamma-\alpha,\gamma).
\end{equation}
Definitely, we want to find a 'measure' for the quantity expression of complex zeros. Next, we will use the 'language' of value distribution theory to do it.

\begin{definition}
Let $f(z)$ be transcendental meromorphic function, whose nonzero zeros are $z_{1}$, $z_{2}$, $\cdots$, $z_{n}$, $\cdots$, counting multiplicities. Let $|z_{n}|=r_{n}$, then $r_{1}\leq r_{2} \leq \cdots \leq r_{n}\leq \cdots$.
We define the infimum of positive number $\tau$ for which $\sum r_{n}^{-\tau}$ converges as the exponent of convergence of zero-sequence, and write it as $\lambda(f)$.
\end{definition}
Let $n(r)$ be the number of nonzero zeros of $f(z)$ in $\overline{D(0,r)}:=\{z:|z|\leq r\}$, as counting functions defined in Section $1$, one can denote
\begin{equation}
N(r)=\int_{0}^{r}\frac{n(t)}{t}dt.
\end{equation}
It is easy to check that
\begin{equation}
n\left(r,\frac{1}{f}\right)=n(r)+n\left(0,\frac{1}{f}\right)
\end{equation}
and
\begin{equation}
N\left(r,\frac{1}{f}\right)=N(r)+n\left(0,\frac{1}{f}\right)\log r.
\end{equation}

\begin{lemma}[\cite{Hayman}]
\label{L:5}
Let $f(z)$ be transcendental meromorphic function, then
\begin{equation}
\lim\sup_{r\rightarrow\infty}\frac{\log^{+}N(r)}{\log r}=\lim\sup_{r\rightarrow\infty}\frac{\log^{+}n(r)}{\log r}=\lambda(f),
\end{equation}
and $\lambda(f)\leq \sigma(f)$.
\end{lemma}
By Theorem \ref{T:2} and Lemma \ref{L:5}, we can get the following result immediately.
\begin{theorem}
Let ${}_{1}F_{1}(\alpha;\gamma;z)$ be a confluent hypergeometric function,
where $\gamma\neq 0,-1, -2, \cdots$. If $\alpha\not\in \mathbb{Z}_{\leq 0}$ and $\gamma-\alpha\not\in \mathbb{Z}_{\leq 0}$, then $\forall \epsilon >0$,
$n\left(r,\frac{1}{f}\right)=n(r)=O(r^{1+\epsilon})$ when $r$ is sufficiently large.
\end{theorem}

\medskip
\section{Uniqueness of ${}_{1}F_{1}(\alpha;\gamma;z)$}
In this section, as applications of Theorem \ref{T:1} and Theorem \ref{T:2}, we shall provide some uniqueness theorems for confluent hypergeometric functions.

Two meromorphic functions $f$ and $g$ are said to share a value $a\in\mathbb{\hat{C}}= \mathbb{C}\bigcup\{\infty\}$ CM (counting multiplicities) if $E_f(a)=E_g(a)$. Here, $E_f(a):=\{z\in\mathbb{C}:f(z)-a=0\}$ denotes the preimage of $a$ under $f$, where a zero of $f-a$ with multiplicity $m$ counts $m$ times in $E_f(a)$. Moreover, $f$ and $g$ are said to share a value $a$ IM (ignoring multiplicities) if $\overline{E}_f(a)=\overline{E}_g(a)$. Here $\overline{E}_f(a)$ denotes the set of the distinct elements in $E_f(a)$, which is called the simplified preimage of $a$ under $f$.
In terms of sharing values, two nonconstant meromorphic functions in $\mathbb{C}$ must be identically equal if they share five values IM, and one must be a M\"{o}bius transform of the other if they share four values CM, the numbers ¡°five¡± and ¡°four¡± are the best possible, as shown by Nevanlinna (see \cite{Hayman}).

\begin{theorem}
\label{T:4}
Let $f(z)$ be an entire function with $f(0)= 1$, and let $g(z):={}_{1}F_{1}(\alpha;\gamma;z)$ be a confluent hypergeometric function. If $f(z)$ and $g(z)$ share $0$ $CM$, and there exists a finite nonzero complex number $a$ such that $\overline{E}_{f}(a)\subseteq \overline{E}_{g}(a)$, and $f'(0)=g'(0)$, then $f(z)\equiv g(z)$.

\end{theorem}

\begin{proof}
By Nevanlinna's Second Main Theorem and Theorem $3.3$, we have
\begin{equation}
\begin{split}
&T(r,f)\leq \overline{N}\left(r,\frac{1}{f}\right)+\overline{N}\left(r,\frac{1}{f-a}\right)+S(r,f)\\
&\leq \overline{N}\left(r,\frac{1}{g}\right)+\overline{N}\left(r,\frac{1}{g-a}\right)+S(r,f)\\
&\leq 2T(r,g)+S(r,f),
\end{split}
\end{equation}
which implies that $\sigma(f)\leq \sigma(g)=1$.

Since $f$ and $g$ share $0$ $CM$ and both are entire, by Lemma \ref{L:4}, we have
\begin{equation}
\frac{f(z)}{g(z)}=e^{Az+B},
\end{equation}
i.e., $f(z)=e^{Az+B}\cdot g(z)$, where $A$ and $B$ are constants.

Note that $f(0)=g(0)=F(\alpha;\gamma;0)=1$, we can get $B=2k\pi i$, $k\in\mathbb{Z}$.
So $f(z)=e^{Az}\cdot g(z)$.
Taking derivative on both sides, we obtain that
\begin{equation}
f'(z)=Ae^{Az}\cdot g(z)+e^{Az}\cdot g'(z).
\end{equation}
Thus, $f'(0)=A\cdot g(0)+g'(0)=A+g'(0)$. By $f'(0)=g'(0)$, we can deduce $A=0$.
Hence, $f(z)\equiv g(z)$.

\end{proof}

The following example shows that the condition '$f'(0)=g'(0)$' in Theorem \ref{T:4} is necessary.

\begin{example}
Let $f(z)=e^{-z}$, $g(z):={}_{1}F_{1}(\alpha;\gamma;z)$ be a confluent hypergeometric function with $\alpha=\gamma$, then $g(z)=e^z$. It is easy to see that $f(z)$ and $g(z)$ share $0$ $CM$, $f(0)=1 $ and
$\overline{E}_{f}(-1)\subseteq \overline{E}_{g}(-1)$, but $f(z)\not\equiv g(z)$.
\end{example}

Next examples given illustrate that the condition '$f(0)=1$' in Theorem \ref{T:4} is needful.

\begin{example}
Let $f(z)=2e^{\frac{z}{2}}$, $g(z):={}_{1}F_{1}(\alpha;\gamma;z)$ be a confluent hypergeometric function with $\alpha=\gamma$, i.e., $g(z)=e^z$. Then $f(z)$ and $g(z)$ share $0$ $CM$, $\overline{E}_{f}(4)\subseteq \overline{E}_{g}(4)$ and $f'(0)=g'(0)$, but $f(z)\not\equiv g(z)$.
\end{example}

\begin{example}
Let $f(z)=-e^{-z}$, $g(z):={}_{1}F_{1}(\alpha;\gamma;z)$ be a confluent hypergeometric function with $\alpha=\gamma$, i.e., $g(z)=e^z$. Then $f(z)$ and $g(z)$ share $0$ $CM$ and $f'(0)=g'(0)$. Furthermore, we have $E_{f}(i)= E_{g}(i)$, but $f(z)\not\equiv g(z)$.
\end{example}

We proceed to showing the following uniqueness result according to the roots of the equation $f'(z)=a$ additionally.
\begin{theorem}
\label{T:5}
Let $f(z)$ be an entire function with $f(0)= 1$, and let $g(z):={}_{1}F_{1}(\alpha;\gamma;z)$ be a confluent hypergeometric function. If $f(z)$ and $g(z)$ share $0$ $CM$, and  $\overline{E}_{f'}(\frac{\alpha}{\gamma})\subseteq \overline{E}_{g'}(\frac{\alpha}{\gamma})$, then $f(z)\equiv g(z)$.
\end{theorem}
\begin{proof}
We distinguish two cases to discuss as follows.

Case I. $\alpha=0$. It is obvious that $g$ is a constant function, by the assumption of the theorem, we obtain
that $f\equiv g$.

Case II. $\alpha\not=0$.

Since
\begin{equation}
\frac{\frac{\alpha}{\gamma}}{f}=\frac{f'}{f}-\frac{f'-\frac{\alpha}{\gamma}}{f''}\frac{f''}{f},
\end{equation}
then
\begin{equation}
m\left(r,\frac{1}{f}\right)\leq m\left(r,\frac{f'}{f}\right)+m\left(r,\frac{f'-\frac{\alpha}{\gamma}}{f''}\right)+m\left(r,\frac{f''}{f}\right)+O(1).
\end{equation}
By the Logarithmic Derivative Lemma and Nevanlinna's first main theorem, we have
\begin{equation}
\begin{split}
&T(r,f)-N\left(r,\frac{1}{f}\right)\\
&\leq m\left(r,\frac{f'-\frac{\alpha}{\gamma}}{f''}\right)+S(r,f)\\
&\leq m\left(r,\frac{f''}{f'-\frac{\alpha}{\gamma}}\right)+N\left(r,\frac{f''}{f'-\frac{\alpha}{\gamma}}\right)
-N\left(r,\frac{f'-\frac{\alpha}{\gamma}}
{f''}\right)+S(r,f)\\
&\leq N\left(r,\frac{f''}{f'-\frac{\alpha}{\gamma}}\right)-N\left(r,\frac{f'-\frac{\alpha}{\gamma}}{f''}\right)+S(r,f)\\
&\leq N(r,f'')+N\left(r,\frac{1}{f'-\frac{\alpha}{\gamma}}\right)-N\left(r,\frac{1}{f''}\right)-N\left(r,f'-\frac{\alpha}{\gamma}\right)
+S(r,f).
\end{split}
\end{equation}
Since $f$ is an entire function such that $f$ and $g$ share $0$ $CM$,  and  $\overline{E}_{f'}(\frac{\alpha}{\gamma})\subseteq \overline{E}_{g'}(\frac{\alpha}{\gamma})$, we obtain
\begin{equation}
\begin{split}
T(r,f)&\leq N\left(r,\frac{1}{f}\right)+N\left(r,\frac{1}{f'-\frac{\alpha}{\gamma}}\right)-N\left(r,\frac{1}{f''}\right)+S(r,f)\\
&\leq N\left(r,\frac{1}{f}\right)+N\left(r,\frac{1}{f'-\frac{\alpha}{\gamma}}\right)-
N\left(r,\frac{1}{(f'-\frac{\alpha}{\gamma})'}\right)+S(r,f)\\
&\leq N\left(r,\frac{1}{f}\right)+\overline{N}\left(r,\frac{1}{f'-\frac{\alpha}{\gamma}}\right)+S(r,f)\\
&\leq N\left(r,\frac{1}{g}\right)+\overline{N}\left(r,\frac{1}{g'-\frac{\alpha}{\gamma}}\right)+S(r,f).
\end{split}
\end{equation}
Note that $m(r, g')\leq m\left(r, \frac{g'}{g}\right)+m(r, g)$, by Theorem \ref{T:2}, we have
\begin{equation}
\begin{split}
T(r,f)&\leq m (r, g)+m(r, g')+S(r,f)\\
&\leq  2T(r,g)+S(r,g)+S(r,f),
\end{split}
\end{equation}
which implies that $\sigma(f)\leq \sigma(g)=1$.

According to the conditions of the theorem, by Lemma \ref{L:4}, we get
\begin{equation}
\frac{f(z)}{g(z)}=e^{Az+B},
\end{equation}
i.e., $f(z)=e^{Az+B}\cdot g(z)$, where $A$ and $B$ are constants.

Note that $g(0)=F(\alpha;\gamma;0)=1$ and $f(0)= 1$, then $B=2k\pi i$, $k\in\mathbb{Z}$.
It follows that $$f(z)=e^{Az}\cdot g(z).$$
Taking derivative on both sides, we obtain that
\begin{equation}
f'(z)=Ae^{Az}\cdot g(z)+e^{Az}\cdot g'(z).
\end{equation}
Thus, $f'(0)=A\cdot g(0)+g'(0)=A+g'(0)$. By $f'(0)=g'(0)=\frac{\alpha}{\gamma}$, we can deduce $A=0$.
Hence, $f(z)\equiv g(z)$.

\end{proof}

The following two examples provided show that some conditions given in Theorem \ref{T:5} are reasonable.

\begin{example}
Let $f(z)=e^{z^{2}}$, $g(z):={}_{1}F_{1}(\alpha;\gamma;z)$ be a confluent hypergeometric function with $\alpha=\gamma$, then $g(z)=e^z$. It is easy to check that $f(z)$ and $g(z)$ share $0$ $CM$ and $f(0)=1$, but $f(z)\not\equiv g(z)$.
So the condition '$\overline{E}_{f'}(\frac{\alpha}{\gamma})\subseteq \overline{E}_{g'}(\frac{\alpha}{\gamma})$' cannot be deleted.
\end{example}

\begin{example}
Let $f(z)=-e^{-z}$, $g(z):={}_{1}F_{1}(\alpha;\gamma;z)$ be a confluent hypergeometric function with $\alpha=\gamma$, i.e., $g(z)=e^z$. Then $f(z)$ and $g(z)$ share $0$ $CM$ and $\overline{E}_{f'}(1)\subseteq \overline{E}_{g'}(1)$. Obviously, $f(z)\not\equiv g(z)$, so the condition '$f(0)=1$' is necessary.
\end{example}

\medskip
\bibliographystyle{amsplain}

\begin{thebibliography}{10}

\bibitem{G. E. Andrews, R. Askey, R. Roy} G. E. Andrews, R. Askey and R. Roy, \textsl{Special Functions},
Cambridge University Press, 1999.

\bibitem{Bateman} H. Bateman, \textsl{Higher Transcendental Functions, Vol.1}, 1953.

\bibitem{GO} A. A. Gol'dberg and I. V. Ostrovskii, \textsl{Value Distribution of Meromorphic Functions}, translations
of Mathematical Monographs, Vol.236, American Mathematical Society, 2008.


\bibitem{Hayman} W. K. Hayman, {\sl Meromorphic Functions}, Clarendon Press, Oxford, 1964.

\bibitem {HX} Y. Z. He and X. Z. Xiao, \textsl{Algebroid Functions and Ordinary Differential Equations}, Science
Press, Beijing 1988 (Chinese).

\bibitem{Ince} E. L. Ince, \textsl{Ordinary Differential Equations}, Dover Publications, 1956.

\bibitem{G. Kristensson} G. Kristensson, \textsl{Second Order Differential Equations: Special Functions and Their
Classification}, Springer New York: New York, NY. 2010.

\bibitem{Laine} I. Laine, \textsl{Nevanlinna Theory and Complex Differential Equations}, Walter de Gruyter, Berlin
1993.

\bibitem{Lebedev} N. N. Lebedev, \textsl{Special Functions and Their Applications}, Rev. English ed. : translated and
edited by Richard A. Silverman, Dover Publications, 1972.

\bibitem{V. Meden and K. Schonhammer} V. Meden and K. Schonhammer, \textsl{Spectral functions for the
Tomonaga-Luttinger model}, Phys. Rev. B, Vol. 46, no. 24, 15753-15760.



\bibitem{P. Nath} P. Nath, \textsl{Confluent hypergeometric function}, The Indian Journal of Statistics (1933-1960),
Vol. 11, No. 2, 153-166.


\bibitem{J. Negro}  J. Negro, L. M. Nieto and O. Rosas-Ortiz,
\textsl{Confluent hypergeometric equations and related solvable potentials in Quantum Mechanics}, J. Math. Phys. 42 (2000), no.12, 7964-7996.

\bibitem{J. B. Seaborn}  J. B. Seaborn, \textsl{Hypergeometric Functions and Their Applications}, Springer-Verlag 1991.

\bibitem{Slater} L. J. Slater, \textsl{Confluent Hypergeometric Functions}, Cambridge University Press, Cambridge-New
York 1960.


\bibitem{N. Ja. Vilenkin}  N. Ja. Vilenkin, \textsl{Special Functions and the Theory of Group Representations}.
American Mathematical Society, Providence, RI.

\bibitem{Wang and Guo} Z. X. Wang and D. R. Guo, \textsl{An Introduction to Special Functions}, Peking University
Press, Beijing 2012 (Chinese), 1968.

\bibitem{Webb} H. A. Webb M.A., and J. R. Airey,  M.A.D.Sc., \textsl{VIII. The practical importance of the confluent
hypergeometric function}, The London, Edinburgh and Dublin philosophical magazine and journal of science, Vol. 36 (1918), 129-141.

\bibitem{Whittaker and Watson} E. T. Whittaker and G. N. Watson, \textsl{A Course of Modern Analysis}, Cambridge
University Press, 1950.



\end{thebibliography}

\end{document}